\documentclass[12pt,reqno]{amsart}

\usepackage{amssymb}

\usepackage{enumitem}

\usepackage{graphicx}

\makeatletter
\@namedef{subjclassname@2010}{%
  \textup{2010} Mathematics Subject Classification}
\makeatother

\usepackage[T1]{fontenc}

\usepackage{mathrsfs}
\usepackage{xcolor}
\usepackage[colorlinks, linkcolor=blue, citecolor=blue, urlcolor=blue]{hyperref}

\newtheorem{theorem}{Theorem}[section]

\newtheorem{fact}[theorem]{Fact}
\newtheorem{lemma}[theorem]{Lemma}

\theoremstyle{definition}
\newtheorem{definition}[theorem]{Definition}

\newcommand{\rstr}{{\upharpoonright}}

\DeclareMathOperator{\scrP}{\mathscr{P}}
\DeclareMathOperator{\fin}{fin}
\DeclareMathOperator{\scrPI}{\mathscr{P}_I}
\DeclareMathOperator{\dom}{dom}
\DeclareMathOperator{\ran}{ran}

\begin{document}


\baselineskip=17pt


\title[A generalized Cantor theorem]{A generalized Cantor theorem in $\mathsf{ZF}$}

\author{Yinhe Peng}
\address{Institute of Mathematics\\
Chinese Academy of Sciences\\
East Zhong Guan Cun Road No.~55\\
Beijing 100190\\
People's Republic of China}
\email{pengyinhe@amss.ac.cn}

\author{Guozhen Shen}
\address{School of Philosophy\\
Wuhan University\\
No.~299 Bayi Road\\
Wuhan\\
Hubei Province 430072\\
People's Republic of China}
\email{shen\_guozhen@outlook.com}

\date{}

\begin{abstract}
It is proved in $\mathsf{ZF}$ (without the axiom of choice) that, for all infinite sets~$M$,
there are no surjections from $\omega\times M$ onto~$\scrP(M)$.
\end{abstract}

\subjclass[2010]{Primary 03E10; Secondary 03E25}

\keywords{$\mathsf{ZF}$, Cantor's theorem, surjection, axiom of choice}

\maketitle

\section{Introduction}\label{s014}
Throughout this paper, we shall work in $\mathsf{ZF}$
(i.e., the Zermelo--Fraenkel set theory without the axiom of choice).

In~\cite{Cantor1892}, Cantor proves that, for all sets~$M$, there are no injections from $\scrP(M)$ into~$M$,
from which it follows easily that, for all sets~$M$, there are no surjections from $M$ onto~$\scrP(M)$.
In~\cite{Specker1954}, Specker proves a generalization of Cantor's theorem,
which states that, for all infinite sets~$M$, there are no injections from $\scrP(M)$ into $M^2$.
In~\cite{Forster2003}, Forster proves another generalization of Cantor's theorem,
which states that, for all infinite sets~$M$, there are no finite-to-one functions from $\scrP(M)$ to~$M$.
In~\cite{Shen2017,Shen2020,Shen2021}, several further generalizations of these results are proved,
among which are the following:
\begin{enumerate}[label=\upshape(\roman*), leftmargin=*, widest=iii]
\item For all infinite sets $M$ and all $n\in\omega$, there are no finite-to-one functions from $\scrP(M)$ to $M^n$ or to~$[M]^n$.\label{s012}
\item For all infinite sets $M$, there are no finite-to-one functions from $\scrP(M)$ to $\omega\times M$.
\item For all infinite sets $M$ and all sets $N$, if there is a finite-to-one function from $N$ to~$M$,
      then there are no surjections from $N$ onto~$\scrP(M)$.\label{s013}
\end{enumerate}

For a set~$M$, let $\fin(M)$ denote the set of all finite subsets of~$M$.
Although it can be proved in $\mathsf{ZF}$ that, for all infinite sets~$M$,
there are no injections from $\scrP(M)$ into $\fin(M)$ (cf.~\cite[Theorem~3]{HalbeisenShelah1994}),
the existence of an infinite set $A$ such that there is a finite-to-one function from $\scrP(A)$ to $\fin(A)$
and such that there is a surjection from $\fin(A)$ onto $\scrP(A)$ is consistent with $\mathsf{ZF}$
(cf.~\cite[Remark~3.10]{Shen2017} and \cite[Theorem~1]{HalbeisenShelah1994}).
Now it is natural to ask whether the existence of an infinite set $A$ such that
there is a surjection from $A^2$ onto $\scrP(A)$ or from $[A]^2$ onto $\scrP(A)$
is consistent with~$\mathsf{ZF}$, and these questions are originally asked in~\cite{Truss1973}
and in~\cite{Halbeisen2018} respectively. In~\cite[Question~5.6]{ShenYuan2020}, it is asked whether
the existence of an infinite set $A$ such that there is a surjection from $\omega\times A$ onto $\scrP(A)$
is consistent with~$\mathsf{ZF}$, and it is noted there that an affirmative answer to this question
would yield affirmative answers to the above two questions. In this paper,
we give a negative answer to this question; that is, we prove in $\mathsf{ZF}$ that,
for all infinite sets~$M$, there are no surjections from $\omega\times M$ onto~$\scrP(M)$.
We also obtain some related results.

\section{Preliminaries}
In this section, we indicate briefly our use of some terminology and notation.
For a function~$f$, we use $\dom(f)$ for the domain of~$f$, $\ran(f)$ for the range of~$f$,
$f[A]$ for the image of $A$ under~$f$, $f^{-1}[A]$ for the inverse image of $A$ under~$f$,
and $f\rstr A$ for the restriction of $f$ to~$A$.
For functions $f$ and~$g$, we use $g\circ f$ for the composition of $g$ and~$f$.
We write $f:A\to B$ to express that $f$ is a function from $A$ to $B$,
and $f:A\twoheadrightarrow B$ to express that $f$ is a function from $A$ \emph{onto} $B$.

\begin{definition}
Let $A,B$ be arbitrary sets.
\begin{enumerate}[leftmargin=*, widest=1]
\item $A\preccurlyeq B$ means that there exists an injection from $A$ into~$B$.
\item $A\preccurlyeq^\ast B$ means that there exists a surjection from a subset of $B$ onto~$A$.
\item $\fin(A)$ denotes the set of all finite subsets of~$A$.
\item $\scrPI(A)$ denotes the set of all infinite subsets of~$A$.
\end{enumerate}
\end{definition}

Clearly, if $A\preccurlyeq B$ then $A\preccurlyeq^\ast B$,
and if $A\preccurlyeq^\ast B$ then $\scrP(A)\preccurlyeq\scrP(B)$ and $\scrPI(A)\preccurlyeq\scrPI(B)$.

\begin{fact}\label{s001}
$\omega_1\preccurlyeq^\ast\scrP(\omega)$.
\end{fact}
\begin{proof}
Cf.~\cite[Theorem~5.11]{Halbeisen2017}.
\end{proof}

In the sequel, we shall frequently use expressions like ``one can explicitly define'' in our formulations,
which is illustrated by the following example.

\begin{theorem}[Cantor-Bernstein]\label{cbt}
From injections $f:A\to B$ and $g:B\to A$,
one can explicitly define a bijection $h:A\twoheadrightarrow B$.
\end{theorem}
\begin{proof}
Cf.~\cite[III.2.8]{Levy1979}.
\end{proof}

\noindent
Formally, Theorem~\ref{cbt} states that one can find a class function $H$ without free variables
such that, whenever $f$ is an injection from $A$ into $B$ and $g$ is an injection from $B$ into~$A$,
$H(f,g)$ is defined and is a bijection of $A$ onto~$B$.

We say that a set $M$ is \emph{Dedekind infinite} if there exists a bijection from $M$ onto some proper subset of~$M$;
otherwise $M$ is \emph{Dedekind finite}. It is well-known that $M$ is Dedekind infinite if and only if there exists
an injection from $\omega$ into~$M$. We say that a set $M$ is \emph{power Dedekind infinite} if the power set of $M$
is Dedekind infinite; otherwise $M$ is \emph{power Dedekind finite}. Recall Kuratowski's celebrated theorem:

\begin{theorem}[Kuratowski]\label{kurt}
A set $M$ is power Dedekind infinite if and only if there exists a surjection from $M$ onto~$\omega$.
\end{theorem}
\begin{proof}
Cf.~\cite[Proposition~5.4]{Halbeisen2017}.
\end{proof}

\section{The main theorem}
In this section, we prove our main theorem, which states that, for all infinite sets~$M$,
there are no surjections from $\omega\times M$ onto~$\scrP(M)$. We first recall the proof of Cantor's theorem.

\begin{theorem}[Cantor]\label{cnt}
From a function $f:M\to\scrP(M)$, one can explicitly define an $A\in\scrP(M)\setminus\ran(f)$.
\end{theorem}
\begin{proof}
It suffices to take $A=\{x\in\dom(f)\mid x\notin f(x)\}$.
\end{proof}

\begin{lemma}\label{s002}
From an infinite ordinal $\alpha$, one can explicitly define an injection $f:\alpha\times\alpha\to\alpha$.
\end{lemma}
\begin{proof}
Cf.~\cite[2.1]{Specker1954} or \cite[IV.2.24]{Levy1979}.
\end{proof}

\begin{lemma}\label{s003}
From an infinite ordinal~$\alpha$, one can explicitly define an injection $f:\fin(\alpha)\to\alpha$.
\end{lemma}
\begin{proof}
Cf.~\cite[Theorem~5.19]{Halbeisen2017}.
\end{proof}

\begin{lemma}\label{s004}
From an infinite ordinal~$\alpha$, one can explicitly define a bijection $f:\omega^\alpha\twoheadrightarrow\alpha$.
\end{lemma}
\begin{proof}
Let $\alpha$ be an infinite ordinal. Let
\[
\exp(\omega,\alpha)=\bigl\{t:\alpha\to\omega\bigm|\{\gamma<\alpha\mid t(\gamma)\neq 0\}\text{ is finite}\bigr\}
\]
and let $r$ be the right lexicographic ordering of $\exp(\omega,\alpha)$.
It is easy to verify that $r$ well-orders $\exp(\omega,\alpha)$ and
the order type of $\langle\exp(\omega,\alpha),r\rangle$ is $\omega^\alpha$ (cf.~\cite[IV.2.10]{Levy1979}).
Let $g$ be the unique isomorphism of $\langle\omega^\alpha,\in\rangle$ onto $\langle\exp(\omega,\alpha),r\rangle$.
Let $h$ be the function on $\exp(\omega,\alpha)$ defined by
\[
h(t)=t\rstr\{\gamma<\alpha\mid t(\gamma)\neq 0\}.
\]
Then $h$ is an injection from $\exp(\omega,\alpha)$ into $\fin(\alpha\times\omega)$.
By Lemmas~\ref{s002} and \ref{s003}, we can explicitly define an injection $p:\fin(\alpha\times\omega)\to\alpha$.
Therefore, $p\circ h\circ g$ is an injection from $\omega^\alpha$ into~$\alpha$.
Now, since the function that maps each $\gamma<\alpha$ to $\omega^\gamma$ is an injection from $\alpha$ into~$\omega^\alpha$,
it follows from Theorem~\ref{cbt} that we can explicitly define a bijection $f:\omega^\alpha\twoheadrightarrow\alpha$.
\end{proof}

\begin{fact}\label{s005}
If $A=B\cup C$ is a set of ordinals which is of order type~$\omega^\delta$,
then either $B$ or $C$ has order type~$\omega^\delta$.
\end{fact}
\begin{proof}
Cf.~\cite[IV.2.22(vii)]{Levy1979}.
\end{proof}

The key step of our proof is the following lemma.

\begin{lemma}\label{s006}
From a surjection $f:\omega\times M\twoheadrightarrow\alpha$, where $\alpha$ is an uncountable ordinal,
one can explicitly define a surjection from $M$ onto~$\alpha$.
\end{lemma}
\begin{proof}
Let $\alpha$ be an uncountable ordinal and let $f$ be a surjection from $\omega\times M$ onto~$\alpha$.
For each $n\in\omega$, let $A_n=f[\{n\}\times M]$, let $\delta_n$ be the order type of~$A_n$,
and let $g_n$ be the unique isomorphism of $\delta_n$ onto~$A_n$.
Let $\delta=\bigcup_{n\in\omega}\delta_n$ and let $g$ be the function on $\omega\times\delta$ defined by
\[
g(n,\gamma)=
\begin{cases}
g_n(\gamma) & \text{if $\gamma<\delta_n$,}\\
0           & \text{otherwise.}
\end{cases}
\]
Then $g$ is a surjection from $\omega\times\delta$ onto~$\alpha$, which implies that $\delta$ is also an uncountable ordinal.
Hence, it follows from Lemma~\ref{s002} that we can explicitly define a surjection from $\delta$ onto~$\alpha$.
So it suffices to explicitly define a surjection from $M$ onto~$\delta$. We consider the following two cases:

\textsc{Case}~1. There exists an $n\in\omega$ such that $\delta=\delta_n$.
Now the function that maps each $x\in M$ to $g_k^{-1}(f(k,x))$ is a surjection of $M$ onto~$\delta$,
where $k$ is the least natural number such that $\delta=\delta_k$.

\textsc{Case}~2. Otherwise. Then $\delta$ is a limit ordinal. Since $\delta>\omega$,
without loss of generality, assume that $\delta_n$ is infinite for all $n\in\omega$.
For each $n\in\omega$, let $\beta_n=\omega^{\delta_n}$. By Lemma~\ref{s004},
for each $n\in\omega$, we can explicitly define a bijection $p_n:\beta_n\twoheadrightarrow\delta_n$.
For each $n\in\omega$, let $h_n$ be the function on $M$ defined by $h_n(x)=p_n^{-1}(g_n^{-1}(f(n,x)))$.
Then, for any $n\in\omega$, $h_n$ is a surjection from $M$ onto~$\beta_n$.
Let $\beta=\omega^\delta$. Clearly, $\beta=\bigcup_{n\in\omega}\beta_n$.
By Lemma~\ref{s004}, it suffices to explicitly define a surjection $h:M\to\beta$.

We first define by recursion two sequences $\langle B_n\rangle_{n\in\omega}$
and $\langle q_n\rangle_{n\in\omega}$ as follows. Let $B_0=M$.
Let $n\in\omega$ and assume that $B_n\subseteq M$ has been defined so that
\begin{equation}\label{s007}
\beta=\bigcup\bigl\{\beta_m\bigm|h_m[B_n]\text{ has order type }\beta_m\bigr\}.
\end{equation}
We define a subset $B_{n+1}$ of $B_n$ and a surjection $q_n:B_n\setminus B_{n+1}\twoheadrightarrow\beta_n$ as follows.
Let $k$ be the least natural number such that $\beta_n<\beta_k$ and such that $h_k[B_n]$ has order type~$\beta_k$.
Let $t$ be the unique isomorphism of $h_k[B_n]$ onto~$\beta_k$. Let $D=\{x\in B_n\mid t(h_k(x))\in\beta_n\}$.
Note that $\beta_n\cdot2<\beta_k$. Now, if \eqref{s007} holds with $B_n$ replaced by~$D$,
we define $B_{n+1}=D$ and let $q_n$ be the function on $B_n\setminus D$ defined by
\[
q_n(x)=
\begin{cases}
\text{the unique $\gamma<\beta_n$ such that $t(h_k(x))=\beta_n+\gamma$} & \text{if $t(h_k(x))<\beta_n\cdot2$,}\\
0                                                                       & \text{otherwise.}
\end{cases}
\]
Otherwise, it follows from \eqref{s007} and Fact~\ref{s005} that \eqref{s007} holds with $B_n$ replaced by $B_n\setminus D$,
and then we define $B_{n+1}=B_n\setminus D$ and let $q_n$ be the function on $D$ defined by $q_n(x)=t(h_k(x))$.
Clearly, in either case, $B_{n+1}\subseteq B_n$, \eqref{s007} holds with $B_n$ replaced by~$B_{n+1}$,
and $q_n$ is a surjection from $B_n\setminus B_{n+1}$ onto~$\beta_n$.
Now, it suffices to define $h=\bigcup_{n\in\omega}q_n\cup(\bigcap_{n\in\omega}B_n\times\{0\})$.
\end{proof}

\begin{lemma}\label{s008}
For all infinite sets $M$ and all sets~$N$, if there is a finite-to-one function from $N$ to~$M$,
then there are no surjections from $N$ onto~$\scrP(M)$.
\end{lemma}
\begin{proof}
Cf.~\cite[Theorem~5.3]{Shen2017}.
\end{proof}

Now we are ready to prove our main theorem.

\begin{theorem}\label{gct}
For all infinite sets~$M$, there are no surjections from $\omega\times M$ onto~$\scrP(M)$.
\end{theorem}
\begin{proof}
Assume toward a contradiction that there exist an infinite set $M$
and a surjection $\Phi:\omega\times M\twoheadrightarrow\scrP(M)$.
We first prove that $M$ is power Dedekind infinite.
Clearly, there is a surjection $\Psi\subseteq\Phi$ from a subset of $\omega\times M$
onto $\scrP(M)$ such that, for all $x\in M$, $\Psi\rstr(\omega\times\{x\})$ is injective.
If $M$ is power Dedekind finite, then $\dom(\Psi)\cap(\omega\times\{x\})$ is finite for all $x\in M$,
and thus there exists a finite-to-one function from $\dom(\Psi)$ to~$M$,
contradicting Lemma~\ref{s008}. Hence, $M$ is power Dedekind infinite.

Now, it follows from Theorem~\ref{kurt} that $\omega\preccurlyeq^\ast M$,
and thus, by Fact~\ref{s001}, $\omega_1\preccurlyeq^\ast\scrP(\omega)\preccurlyeq\scrP(M)\preccurlyeq^\ast\omega\times M$,
which implies that $\omega_1\preccurlyeq^\ast M$ by Lemma~\ref{s006} and hence $\omega_1\preccurlyeq\scrP(M)$.
Let $h$ be an injection from $\omega_1$ into~$\scrP(M)$. In what follows,
we get a contradiction by constructing by recursion an injection $H$
from $\mathrm{Ord}$ (the class of all ordinals) into~$\scrP(M)$.

For $\gamma<\omega_1$, we take $H(\gamma)=h(\gamma)$.
Now, we assume that $\alpha$ is an uncountable ordinal and that $H\rstr\alpha$ is an injection from $\alpha$ into~$\scrP(M)$.
Then $(H\rstr\alpha)^{-1}\circ\Phi$ is a surjection from a subset of $\omega\times M$ onto $\alpha$
and hence can be extended by zero to a surjection $f:\omega\times M\twoheadrightarrow\alpha$.
By Lemma~\ref{s006}, $f$ explicitly provides a surjection $g:M\twoheadrightarrow\alpha$.
Since $(H\rstr\alpha)\circ g$ is a surjection from $M$ onto~$H[\alpha]$,
it follows from Theorem~\ref{cnt} that we can explicitly define an
$H(\alpha)\in\scrP(M)\setminus H[\alpha]$ from $H\rstr\alpha$ (and~$\Phi$).
\end{proof}

\section{A further generalization}
In~\cite{Kirmayer1981}, Kirmayer proves that, for all infinite sets~$M$,
there are no surjections from $M$ onto~$\scrPI(M)$. In this section,
we generalize this result by showing that, for all infinite sets~$M$,
there are no surjections from $\omega\times M$ onto~$\scrPI(M)$,
which is also a generalization of Theorem~\ref{gct}.
The proof is similar to that of Theorem~\ref{gct}, but first we have to
prove that Lemma~\ref{s008} holds with $\scrP(M)$ replaced by~$\scrPI(M)$.

\begin{lemma}\label{s009}
From an infinite ordinal~$\alpha$, one can explicitly define an injection $f:\scrP(\alpha)\to\scrPI(\alpha)$.
\end{lemma}
\begin{proof}
By Lemma~\ref{s002}, we can explicitly define an injection $p:\alpha\times\alpha\to\alpha$.
Let $f$ be the function on $\scrP(\alpha)$ defined by
\[
f(A)=
\begin{cases}
p[A\times\{0\}]                   & \text{if $A$ is infinite,}\\
p[(\alpha\setminus A)\times\{1\}] & \text{otherwise.}
\end{cases}
\]
Then it is easy to see that $f$ is an injection from $\scrP(\alpha)$ into~$\scrPI(\alpha)$.
\end{proof}

\begin{lemma}\label{s010}
From a set $M$, a finite-to-one function $f:N\to M$, and a surjection $g:N\twoheadrightarrow\alpha$,
where $\alpha$ is an infinite ordinal, one can explicitly define a surjection $h:M\twoheadrightarrow\alpha$.
\end{lemma}
\begin{proof}
Cf.~\cite[Lemma~5.2]{Shen2017}.
\end{proof}

\begin{lemma}\label{s011}
For all infinite sets $M$ and all sets~$N$, if there is a finite-to-one function from $N$ to~$M$,
then there are no surjections from $N$ onto~$\scrPI(M)$.
\end{lemma}
\begin{proof}
Assume toward a contradiction that there exist an infinite set $M$ and a set $N$ such that
there exist a finite-to-one function $f:N\to M$ and a surjection $\Phi:N\twoheadrightarrow\scrPI(M)$.
Clearly, the function that maps each cofinite subset $A$ of $M$ to the cardinality of $M\setminus A$
is a surjection from a subset of $\scrPI(M)$ onto~$\omega$,
and hence $\omega\preccurlyeq^\ast\scrPI(M)\preccurlyeq^\ast N$,
which implies that $\omega\preccurlyeq^\ast M$ by Lemma~\ref{s010}.
Thus $\omega\preccurlyeq\scrPI(\omega)\preccurlyeq\scrPI(M)$.
Let $h$ be an injection from $\omega$ into~$\scrPI(M)$.
In what follows, we get a contradiction by constructing by recursion
an injection $H$ from $\mathrm{Ord}$ into~$\scrPI(M)$.

For $n\in\omega$, we take $H(n)=h(n)$. Now, we assume that $\alpha$ is an infinite
ordinal and that $H\rstr\alpha$ is an injection from $\alpha$ into~$\scrPI(M)$.
Then $(H\rstr\alpha)^{-1}\circ\Phi$ is a surjection from a subset of $N$ onto $\alpha$
and hence can be extended by zero to a surjection $g:N\twoheadrightarrow\alpha$. By Lemma~\ref{s010},
from $M$, $f$, and~$g$, we can explicitly define a surjection $p:M\twoheadrightarrow\alpha$.
Then the function $q$ on $\scrPI(\alpha)$ defined by $q(A)=p^{-1}[A]$ is an injection
from $\scrPI(\alpha)$ into~$\scrPI(M)$, and thus it follows from Lemma~\ref{s009}
that we can explicitly define an injection $t:\scrP(\alpha)\to\scrPI(M)$.
Then $t^{-1}\circ(H\rstr\alpha)$ is a bijection from a subset of $\alpha$ onto $t^{-1}[H[\alpha]]$
and thus can be extended by zero to a function $u:\alpha\to\scrP(\alpha)$.
By Theorem~\ref{cnt}, we can explicitly define a $B\in\scrP(\alpha)\setminus\ran(u)$.
Since $t^{-1}[H[\alpha]]\subseteq\ran(u)$, it follows that $B\notin t^{-1}[H[\alpha]]$,
which implies that $t(B)\notin H[\alpha]$. Now, it suffices to define $H(\alpha)=t(B)$.
\end{proof}

We are now in a position to prove the result mentioned at the beginning of this section.

\begin{theorem}\label{gkt}
For all infinite sets~$M$, there are no surjections from $\omega\times M$ onto~$\scrPI(M)$.
\end{theorem}
\begin{proof}
We proceed along the lines of the proof of Theorem~\ref{gct}.
Assume toward a contradiction that there exist an infinite set $M$
and a surjection $\Phi:\omega\times M\twoheadrightarrow\scrPI(M)$.
We first prove that $M$ is power Dedekind infinite.
Clearly, there is a surjection $\Psi\subseteq\Phi$ from a subset of $\omega\times M$
onto $\scrPI(M)$ such that, for all $x\in M$, $\Psi\rstr(\omega\times\{x\})$ is injective.
If $M$ is power Dedekind finite, then $\dom(\Psi)\cap(\omega\times\{x\})$ is finite for all $x\in M$,
and thus there exists a finite-to-one function from $\dom(\Psi)$ to~$M$,
contradicting Lemma~\ref{s011}. Hence, $M$ is power Dedekind infinite.

Now, it follows from Theorem~\ref{kurt} that $\omega\preccurlyeq^\ast M$,
and thus, by Fact~\ref{s001} and Lemma~\ref{s009},
$\omega_1\preccurlyeq^\ast\scrP(\omega)\preccurlyeq\scrPI(\omega)\preccurlyeq\scrPI(M)\preccurlyeq^\ast\omega\times M$,
which implies that $\omega_1\preccurlyeq^\ast M$ by Lemma~\ref{s006}
and hence $\omega_1\preccurlyeq\scrPI(\omega_1)\preccurlyeq\scrPI(M)$.
Let $h$ be an injection from $\omega_1$ into~$\scrPI(M)$. In what follows,
we get a contradiction by constructing by recursion
an injection $H$ from $\mathrm{Ord}$ into~$\scrPI(M)$.

For $\gamma<\omega_1$, we take $H(\gamma)=h(\gamma)$.
Now, we assume that $\alpha$ is an uncountable ordinal and that $H\rstr\alpha$ is an injection from $\alpha$ into~$\scrPI(M)$.
Then $(H\rstr\alpha)^{-1}\circ\Phi$ is a surjection from a subset of $\omega\times M$ onto $\alpha$
and hence can be extended by zero to a surjection $f:\omega\times M\twoheadrightarrow\alpha$.
By Lemma~\ref{s006}, $f$ explicitly provides a surjection $g:M\twoheadrightarrow\alpha$.
Then the function $q$ on $\scrPI(\alpha)$ defined by $q(A)=g^{-1}[A]$ is an injection
from $\scrPI(\alpha)$ into~$\scrPI(M)$, and thus it follows from Lemma~\ref{s009}
that we can explicitly define an injection $t:\scrP(\alpha)\to\scrPI(M)$.
Then $t^{-1}\circ(H\rstr\alpha)$ is a bijection from a subset of $\alpha$ onto $t^{-1}[H[\alpha]]$
and thus can be extended by zero to a function $u:\alpha\to\scrP(\alpha)$.
By Theorem~\ref{cnt}, we can explicitly define a $B\in\scrP(\alpha)\setminus\ran(u)$.
Since $t^{-1}[H[\alpha]]\subseteq\ran(u)$, it follows that $B\notin t^{-1}[H[\alpha]]$,
which implies that $t(B)\notin H[\alpha]$. Now, it suffices to define $H(\alpha)=t(B)$.
\end{proof}

Using the method presented here, we can also show that the statements \ref{s012}--\ref{s013} in Section~\ref{s014}
holds with $\scrP(M)$ replaced by~$\scrPI(M)$ (Lemma~\ref{s011} is just the statement \ref{s013} for~$\scrPI(M)$).
We shall omit the details.

The questions whether the existence of an infinite set $A$ such that
there is a surjection from $A^2$ onto $\scrP(A)$ or from $[A]^2$ onto $\scrP(A)$
is consistent with~$\mathsf{ZF}$ are left open.

\subsection*{Acknowledgements}
Peng was partially supported by NSFC No.~11901562 and the Hundred Talents Program of the Chinese Academy of Sciences.
Shen was partially supported by NSFC No.~12101466.


\end{document}